\documentclass{article}

\usepackage[a-1b]{pdfx}
\usepackage{amssymb, amsfonts, amsmath, tikz, tikz-cd, graphicx}
\usepackage[english]{babel}
\usepackage{amsthm}
\usepackage{longtable}
\usepackage{tikzit}

\tikzstyle{BasicNode}=[fill=none, draw=black, shape=circle]
\tikzstyle{Black Node}=[fill=black, draw=black, shape=circle, tikzit fill=black]
\tikzstyle{Red Node}=[fill=red, draw=black, shape=circle, tikzit fill=red]
\tikzstyle{Blue Node}=[fill=blue, draw=black, shape=circle, tikzit fill=blue]
\tikzstyle{Orange Node}=[fill=orange, draw=black, shape=circle, tikzit fill=orange]
\tikzstyle{Yellow Node}=[fill=yellow, draw=black, shape=circle, tikzit fill=yellow]
\tikzstyle{Green Node}=[fill=green, draw=black, shape=circle, tikzit fill=green]
\tikzstyle{L Blue Node}=[fill=cyan, draw=black, shape=circle, tikzit fill=cyan]
\tikzstyle{Purple Node}=[fill={rgb,255: red,128; green,0; blue,128}, draw=black, shape=circle, tikzit fill={rgb,255: red,128; green,0; blue,128}]

\tikzstyle{Arrow}=[->]
\tikzstyle{Blue Arrow}=[draw=blue, ->, tikzit draw=blue]
\tikzstyle{Big Blue Line}=[-, ultra thick, draw=blue]
\tikzstyle{Big Blue Arrow}=[draw=blue, ->, ultra thick]
\tikzstyle{red edge}=[-, draw=red, tikzit draw=red]
\tikzstyle{Orange}=[-, fill=none, draw={rgb,255: red,255; green,128; blue,0}, tikzit draw={rgb,255: red,255; green,128; blue,0}]
\tikzstyle{Yellow}=[-, draw=yellow, tikzit draw=yellow]
\tikzstyle{Green}=[-, draw=green]
\tikzstyle{L Blue}=[-, draw=cyan]
\tikzstyle{Blue}=[-, draw=blue, tikzit draw=blue]
\tikzstyle{Purple}=[-, draw={rgb,255: red,128; green,0; blue,128}, tikzit draw={rgb,255: red,128; green,0; blue,128}]
\tikzstyle{Big Red arrow}=[draw=red, ->, ultra thick]
\tikzstyle{Green arrow}=[draw=green, ->]
\tikzstyle{Purple arrow}=[draw={rgb,255: red,128; green,0; blue,128}, ->]
\tikzstyle{Thick Line}=[-, ultra thick]
\tikzstyle{Dashed blue}=[-, draw=blue, dashed]
\tikzstyle{Big red line}=[-, draw=red, ultra thick]
\tikzstyle{Red arrow}=[draw=red, tikzit draw=red, ->]
\tikzstyle{red dashed}=[-, draw=red, dashed]

\usepackage{listings}
\usepackage{color}

\definecolor{codegreen}{rgb}{0,0.6,0}
\definecolor{codegray}{rgb}{0.5,0.5,0.5}
\definecolor{codepurple}{rgb}{0.58,0,0.82}
\definecolor{backcolour}{rgb}{0.95,0.95,0.92}

\lstdefinestyle{mystyle}{
    backgroundcolor=\color{backcolour},   
    commentstyle=\color{codegreen},
    keywordstyle=\color{magenta},
    numberstyle=\tiny\color{codegray},
    stringstyle=\color{codepurple},
    basicstyle=\ttfamily\footnotesize,
    breakatwhitespace=false,         
    breaklines=true,                 
    captionpos=b,                    
    keepspaces=true,                 
    numbers=left,                    
    numbersep=5pt,                  
    showspaces=false,                
    showstringspaces=false,
    showtabs=false,                  
    tabsize=2
}

\newcommand{\F}{\mathcal{F}}
\newcommand{\Q}{\hat Q}

\newtheorem{thm}{Theorem}[section]
\newtheorem{prop}[thm]{Proposition}
\newtheorem{cor}[thm]{Corollary}
\newtheorem{lem}[thm]{Lemma}
\theoremstyle{definition}
\newtheorem{defn}[thm]{Definition}

\newtheorem{obs}[thm]{Observation}

\newtheorem{conj}[thm]{Conjecture}

\newtheorem{ques}[thm]{Question}

\title{The Zero Forcing Number of Twisted Hypercubes}
\author{Collier, Peter \\\texttt{peter.collier@dal.ca}
\and
Janssen, Jeannette \\\texttt{jeannette.janssen@dal.ca}}
\date{\today}

\begin{document}

\maketitle

\begin{abstract}
    Twisted hypercubes are graphs that generalize the structure of the hypercube by relaxing the symmetry constraint while maintaining degree-regularity and connectivity. We study the zero forcing number of twisted hypercubes. Zero forcing is a graph infection process in which a particular colour change rule is iteratively applied to the graph and an initial set of vertices. We use the alternative framing of forcing arc sets to construct a family of twisted hypercubes of dimension k$\geq 3$ with zero forcing sets of size $2^{k-1}-2^{k-3}+1$, which is below the minimum zero forcing number of the hypercube.
\end{abstract}

\section{Introduction}

\label{intro}

In the realm of combinatorics and computer science, the hypercube stands out as a compelling and versatile structure. Hypercubes extend the geometric notion of a cube into higher dimensions, providing a robust framework for exploring complex problems in multidimensional spaces. These structures offer significant insights across various fields, including transport network design \cite{qbe_traffic}, quantum optimization \cite{quant_qbe}, and data organization \cite{data_qbe}.

The twisted hypercube generalizes the standard hypercube structure by relaxing the symmetry constraint while maintaining the degree-regularity and connectivity of the graphs. With this additional level of freedom, we seek to determine whether the underlying structure of the network can be better tailored to a specific  process in order to optimize its performance. We will look at one particular graph infection process called zero forcing.

Zero forcing, introduced in \cite{minrank}, is a graph infection process where every vertex of a graph is coloured either white or blue. An initial set of blue vertices will propagate according to a colour change rule, defined in Section \ref{prelims}, until no further changes are possible. If all vertices in the graph eventually turn blue, then the initial blue set is called a zero forcing set. The goal of studying zero forcing is to find the smallest possible zero forcing set, whose size is called the zero forcing number of the graph. 

Zero forcing was initially introduced in \cite{minrank} as an upper bound for the maximum nullity of matrices with certain patterns of non-zero entries. In  \cite{minrank}, the zero forcing number of various families of graphs  was determined, such as supertriangles, hypercubes, the line graphs as well as graph products of some graph families.
Zero forcing has also found applications in discrete mathematics (see inverse eigenvalue problem \cite{zfiep}), applied mathematics (see PMU placement problem \cite{zfpmu}), and control theory (see quantum control problem \cite{zfqcp}). For more on zero forcing and its applications, see \cite{hogben2022inverse}. The zero forcing number of graphs is an interesting field of research in its own right. Recently, the zero forcing number for families of graphs, such as Generalized Johnson, Generalized Grassmann, and Hamming graphs \cite{johngrassham} were found. Additionally, characterizations of graphs with certain degree-related zero forcing numbers were given in \cite{LIANG202381}. The zero forcing number of random graphs and random regular graphs is studied in \cite{Kalinowski19} and \cite{bal2018zeroforcingnumberrandom}, respectively. In \cite{He2024} a bound for the zero forcing number of claw-free cubic graphs is given, answering a problem posed in \cite{Davila2020}.

A common theme in zero forcing research is finding minimal zero forcing sets by using the structure of the graphs, while determining minimality using the maximum nullity of the graph. Both finding zero forcing sets and determining maximum nullity are known to be NP-Hard problems \cite{np-hard, dir_np-hard}. 
There are some combinatorial lower bounds on the zero forcing number, such as the minimum degree or the clique cover number. However, these are trivial in almost all graphs and can be arbitrarily far from the actual zero forcing number.

Recently, a new method of determining whether an initial set is a zero forcing set was developed in \cite{chaintwist}, which makes constructing and verifying such sets much simpler. In this paper, we use this method to find an upper bound on the zero forcing number of a family of twisted hypercubes.

\subsection{Preliminaries}
\label{prelims}

Given a graph, $G=(V,E)$, where each vertex is coloured either white or blue, and an initial set of blue vertices, $S \subseteq V$, we define zero forcing as the graph infection process in which we iteratively apply the following colour change rule:\\
{\it If a blue vertex, $u$, has exactly one white neighbour, $v$, then $v$ changes from white to blue.}

In this case, we say that \emph{$u$ forces $v$}, and denote this relation by $u\rightarrow v$. The derived set is the set of blue vertices after performing all possible forces. If the derived set is the entire vertex set, then we call $S$ a zero forcing set. The size of the smallest zero forcing set is called the zero forcing number of $G$, and is denoted $Z(G)$. 

It was noted in \cite{zfpar} that an observation of the zero forcing process with initial blue set $S$ can be described by a set of directed edges, or {\it  arcs}. 

\begin{defn}
\label{arc_set}
    For a graph $G=(V,E)$, we will call a set $\F\subset V\times V$ an {\it arc set of $G$} provided that for any ordered pair $(u,v)\in\F$, the edge $uv\in E$ and $(v,u)\notin\F$.
\end{defn}

For a given zero forcing process, we can represent the relation $u\rightarrow v$ as an arc set. These arcs form directed paths, or {\it chains}, $v_1v_2\dots v_k$, such that $v_1\in S$, and $(v_i,v_{i+1})\in\F$ is a directed edge for $1\leq i<k$. The arc $(v_i,v_{i+1})\in\F$, represents the vertex $v_i$ changing $v_{i+1}$ from white to blue, or $v_i\rightarrow v_{i+1}$, in the zero forcing process. A characterization of zero forcing sets in the context of {\it forcing arc sets} was given in \cite{zfpar, chaintwist}. This characterization relies on a particular structure that may arise in the formation of an arc set called a {\it chain twist}.

\begin{defn}
\label{ctdef}
    Given a graph $G$ and arc set $\F$, a chain twist in $\F$ is a cycle, $v_1v_2\dots v_kv_1$ with the property
    \begin{equation}
    \label{chainprop}
        (v_i,v_{i+1})\notin\F\Rightarrow (v_{i-1},v_i)\in\F, \mbox{and } (v_{i+1},v_{i+2})\in\F,
    \end{equation}
    for all $1\leq i\leq k$, where addition is done modulo $k$.
\end{defn}

If a cycle in $\F$ satisfies Definition \ref{ctdef}, then we will say that $\F$ {\it contains a chain twist}. It was shown in \cite{chaintwist} that the absence of any chain twists in an arc set $\F$ is necessary and sufficient to conclude that $\F$ is a forcing arc set.

\begin{thm}
\label{ctthm}
    \cite{chaintwist} Let $G$ be a graph and $\F$ an arc set of $G$. Then $\F$ is a forcing arc set if and only if $\F$ does not contain a chain twist.
\end{thm}

We will denote a graph $G$ along with a forcing arc set $\F$ by $(G,\F)$. As forcing arc sets also correspond to a completed zero forcing process for some initial set of blue vertices, $S$, Theorem \ref{ctthm} tells us that any arc set without a chain twist also produces a zero forcing set for $G$. The zero forcing set consists of the initial vertices of each chain. From the size of $\F$, we can calculate the size of the corresponding zero forcing set since $\F$ is a spanning forest of dipaths. These two values are related by
\begin{equation}
\label{arc_to_zf}
    |S|=|V|-|\F|
\end{equation}
In other words, the number of chains in $\F$ is the size of the zero forcing set. 

The Cartesian product of two graphs, $G$ and $H$, denoted $G\square H$, is the graph with vertex set $V(G)\times V(H)$ and edges 
\begin{align*}
    (u,x)(v,y)\in E(G\square H) \iff& uv\in E(G)\text{ and }x=y,\text{ or }\\
        &u=v\text{ and }xy\in E(H).
\end{align*}

In \cite{minrank}, the Cartesian product was shown to provide a bound on the zero forcing number of the product graph in terms of the factors.

\begin{thm}
\label{cpbound}
    \cite{minrank} Let $G$, $H$ be graphs. Then $$Z(G\square H)\leq\min\{Z(G)|H|,Z(H)|G|\}.$$
\end{thm}

\begin{figure}[!t]
    \centering
    \includegraphics[scale=0.5]{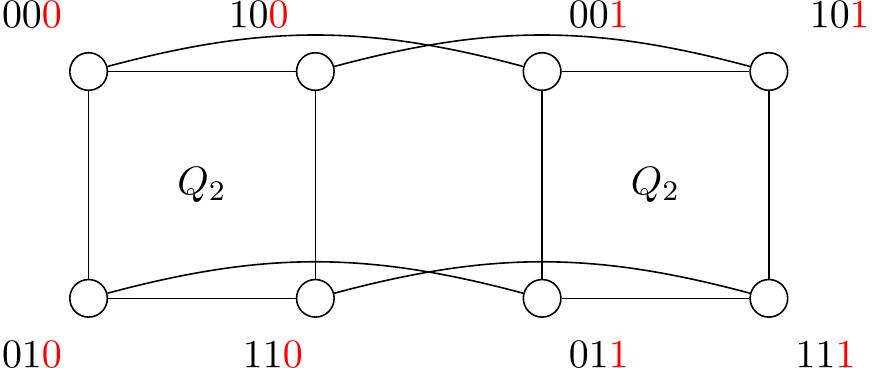}
    \caption{An example of constructing $Q_3$ from two copies of $Q_2$ using the standard matching. Appended digits indicated in red.}
    \label{Q3demo}
\end{figure}

The $n$-dimensional hypercube, $Q_n$, is the graph with vertex set $\{0,1\}^n$, all bit-strings of length $n$, where vertices are adjacent exactly when their bit-strings differ in exactly one position. We can also construct $Q_n$ inductively for $n\geq1$ as $Q_n=Q_{n-1}\square K_2$, where $Q_0$ is a single vertex. This can be thought of as taking two copies of $Q_{n-1}$, appending 0 to the end of all vertices in one copy, appending 1 to all vertices of the other copy, and forming a matching between the copies that connects all vertices that differ in only the last bit, illustrated in Figure \ref{Q3demo}. 

The zero forcing number of the hypercube was determined in \cite{minrank} to be $Z(Q_n)=2^{n-1}$. The vertices of one copy of $Q_{n-1}$ form a zero forcing set for $Q_n$. Since all neighbours of blue vertices in the same copy of $Q_{n-1}$ will be blue vertices, the neighbour in the other copy will be the only white neighbour and will, therefore, be forced. So, a zero forcing set with half of the vertices of $Q_n$ can always be constructed this way. The authors of \cite{minrank} found a lower bound on the maximum nullity of the hypercube graph of $2^{n-1}$ in order to show that this is the best possible. 

Twisted hypercubes are a generalization of the standard hypercube graph, first introduced in \cite{twistcube}, that have been studied for their high connectivity and low diameter \cite{thqbe}. The unique 0-dimensional twisted hypercube is the graph that consists of a single vertex, represented by the empty string. An $n$-dimensional twisted hypercube is a graph formed by taking two $(n-1)$-dimensional twisted hypercubes, not necessarily the same, appending 0 to the end of all vertices in one copy, appending 1 to all vertices of the other copy, and forming any matching between the copies. An example is given in Figure \ref{thqdemo}. We will refer to the matching in which all edges consist of vertices which differ in only the final bit as the {\it standard matching}, and any edges whose endpoints differ in more than one bit as {\it twisted edges}. For a vertex $v$, we will refer to the unique vertex $u\in N(v)$ that differs in the final bit as the {\it twin of $v$}. Note that the twin of a vertex in a twisted hypercube could differ in more than just the final position of the bit string, and every vertex has exactly one twin.

\begin{figure}[t]
    \centering
    \includegraphics[scale=0.5]{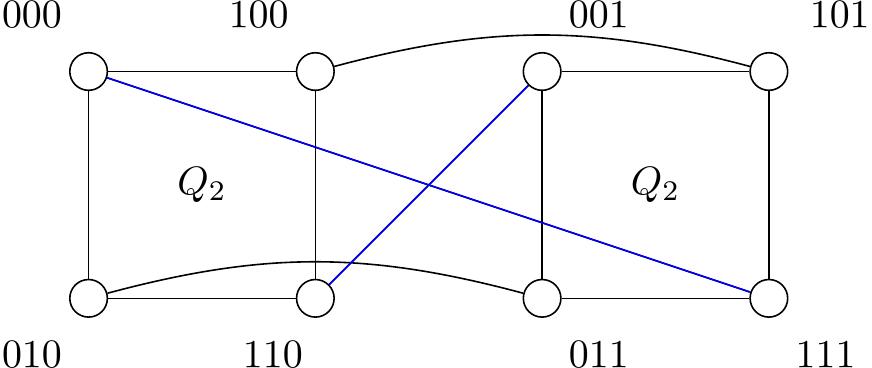}
    \caption{An example of a 3-dimensional twisted hypercube constructed from two copies of $Q_2$ with twisted edges indicated in blue.}
    \label{thqdemo}
\end{figure}

While the zero forcing number of the hypercube was determined to be $Z(Q_n)=2^{n-1}$, the zero forcing number of twisted hypercubes is not known in general. All twisted hypercubes of dimension $n$ have zero forcing number at most $2^{n-1}$, as taking one of the lower dimensional twisted hypercubes as our initial set is always a zero forcing set. In \cite{mythesis}, we show that no twisted hypercube of dimension less than 4 has zero forcing number smaller than that of the same size hypercube. By brute force computation, we were able to find twisted hypercubes of dimension 4, 5, and 6 with zero forcing number 7, 13, and 25 respectively and establish by computation that no twisted hypercube of the same dimension has smaller zero forcing number. Using Theorem \ref{cpbound} we found, in \cite{mythesis}, a family of twisted hypercubes with zero forcing number at most $2^{n-1}-2^{n-4}-2^{n-5}-2^{n-6}$. These results hint towards a family of twisted hypercubes whose zero forcing numbers follow this pattern of doubling and subtracting one. The primary result of this paper is to prove this result, conjectured in \cite{mythesis}:

\begin{thm}
    \label{zfbound}
    There exists a family of twisted hypercubes with zero forcing number that is bounded above by $2^{n-1}-2^{n-3}+1$ for $n\geq3$.
\end{thm}

We will refer to this family of twisted hypercubes as the {\it minority cubes}, which we will denote by $\Q_n$. To prove Theorem \ref{zfbound},  we will construct the minority cubes in Section \ref{const}, and Section \ref{proof} will consist of proving that these twisted hypercubes have the desired property.

In this section, we begin by constructing the 4-dimensional minority cube, $\Q_4$, and the arc set $\F_4$ for illustrative purposes. We will then outline how to construct the family $\Q_n$. Moreover, we will construct the arc set $\F_n$ in $\Q_n$. Later, in Observation \ref{num_arcs}, we show that $\F_n$ has the correct size, and in Theorem \ref{ct_thm} we show that $\F_n$ does not contain a chain twist. In the arc set, we will denote the arc $uv$ as $u\rightarrow v$. We will denote an edge in the graph that is not an arc $xy$ as $x-y$.
\section{Construction}

\label{construction}

In this section, we begin by constructing the 4-dimensional minority cube, $\Q_4$, and the arc set $\F_4$ for illustrative purposes. We will then outline how to construct the family $\Q_n$. Moreover, we will construct the arc set $\F_n$ in $\Q_n$. Later, in Observation \ref{num_arcs}, we show that $\F_n$ has the correct size, and in Theorem \ref{ct_thm} we show that $\F_n$ does not contain a chain twist. In the arc set, we will denote the arc from $u$ to $v$ as $u\rightarrow v$. We will denote an edge with endpoints $x$ and $y$ in the graph that is not an arc as $x-y$.

\subsection{Constructing $(\Q_4,\F_4)$}

For this construction, we develop the minority cubes and the forcing arc sets beginning with dimension $n=3$ and iteratively extending this construction to higher dimensions. As we are only concerned with twisted hypercubes of dimension $\geq3$, we begin our construction with a copy of $\Q_3=Q_3$, being the 3-dimensional cube, with arc set $\F_3$ consisting of four chains, two of which are isolated vertices. 

\begin{figure}
    \centering
    \includegraphics[scale=0.5]{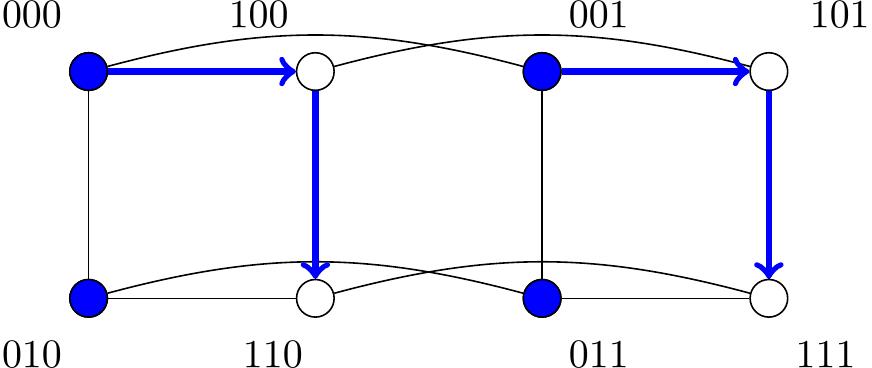}
    \caption{$(\Q_3,\F_3)$. The arc set $\F_3$ drawn on the graph $Q_3$.}
    \label{Q3zfs}
\end{figure}

These chains, as seen in Figure \ref{Q3zfs}, are 
\begin{align*}
    000&\rightarrow100\rightarrow110\\
    010&\\
    001&\rightarrow101\rightarrow111\\
    011&.
\end{align*}
This forms the base of our construction, $(\Q_3, \F_3)$.

\begin{figure}[ht]
    \centering
    \includegraphics[scale=0.5]{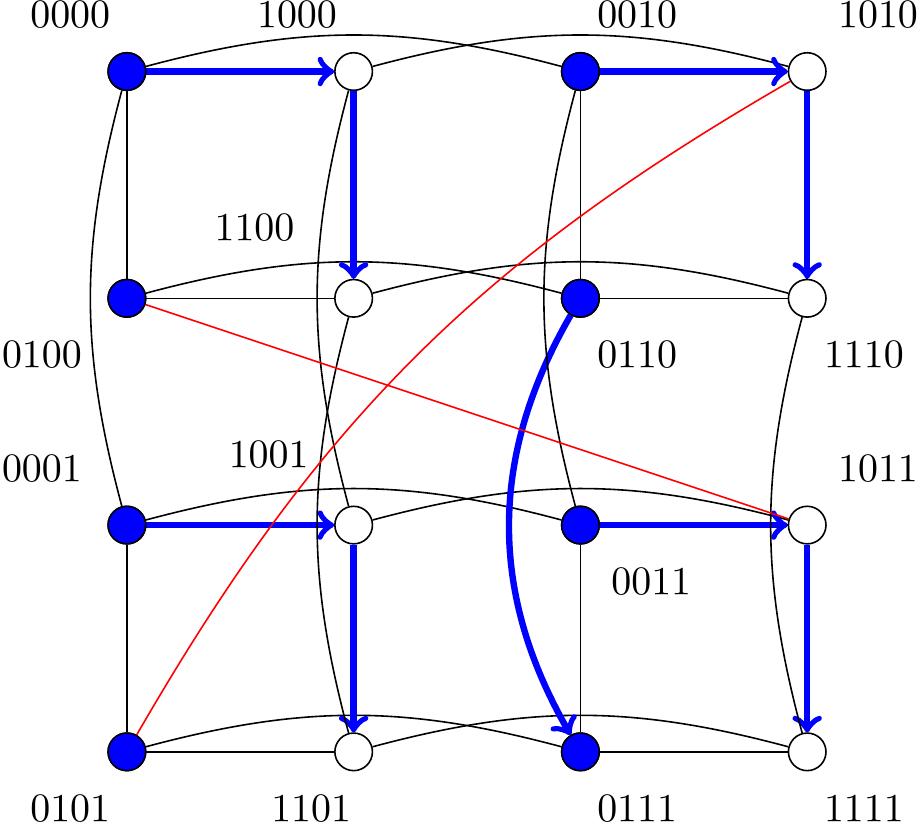}
    \caption{The forcing arc set $\F_4$ indicated on the graph $\Q_4$ with twisted edges shown in red.}
    \label{Q4zfs}
\end{figure}

To construct $\Q_4$, we consider two copies of $\Q_3$ where we append 0 to one copy's bit strings and 1 to the other's. The matching between the two copies will follow the standard matching for all but two edges. These twisted edges will be $0100-1011$ and $1010-0101$. To construct $\F_4$, we take the union of the copies of $\F_3$ in each of the copies of $\Q_3$. We now add one additional arc between the two copies of $(\Q_3,\F_3)$ from $0110\rightarrow0111$. Call this the {\it bridge arc}. Note that these vertices are adjacent in our matching, neither vertex is incident to one of the twisted edges, and they differ in only the final bit. The vertices 0110 and 0111 are both isolated in their respective copies of $\F_3$, so we are able to add this arc without impacting any other chains. This results in a 4-dimensional twisted hypercube along with a forcing arc set with $9=2|\F_3|+1$ arcs and two isolated vertices. This will be our $(\Q_4,\F_4)$. See Figure \ref{Q4zfs}.

\subsection{Construction for $n>4$}
\label{const}

To construct $(\Q_n,\F_n)$, $n\geq4$, we begin with two copies of $(\Q_{n-1},\F_{n-1})$ where we append 0 to one copy's bit strings and 1 to the other's. Let $\bar0$ be the string of all zeroes of length $n-4$. Note that $\bar0$ could be the empty string. The matching between the two copies will follow the standard matching for all but two edges. These twisted edges will be $01\bar{0}00-10\bar{0}11$ and $10\bar{0}10-01\bar{0}01$. We now add the bridge arc between the two copies of $(\Q_{n-1},\F_{n-1})$, $01\bar{0}10\rightarrow01\bar{0}11$. This will be our $(\Q_{n},\F_{n})$, illustrated in Figure \ref{Qnzfs}. Note that this construction also works for constructing $\Q_4$ from copies of $\Q_3$.

\begin{lem}
\label{rel_num_arcs}
    $(\Q_{n},\F_{n})$ has $2|\F_{n-1}|+1$ arcs and two isolated vertices.
\end{lem}

\begin{proof}
    Each copy of $(\Q_{n-1},\F_{n-1})$ contributes $|\F_{n-1}|$ arcs and two isolated vertices. The bridge arc is formed between two isolated vertices, resulting in $2|\F_{n-1}|+1$ arcs and leaving two isolated vertices.
\end{proof}

\begin{figure}[!t]
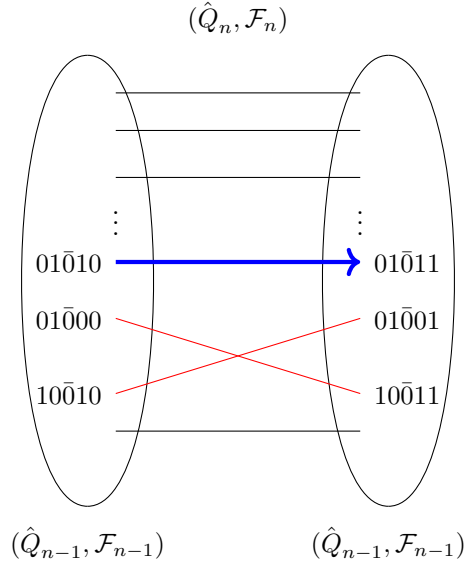

    \ctikzfig{hatQn}
    \caption{How the forcing arc set $\F_n$ is constructed from two copies of $(\Q_{n-1},\F_{n-1})$}
    \label{Qnzfs}
\end{figure}

To discuss these arcs more generally, we note that the first two bits in the tail of the arc determine the behaviour of the arc. So we can define substrings of our vertex labels to help us discuss these arcs. Let $a\in\{0,1\}^{n-2}$ be any bit string of length $n-2$. Then $\F_n$ contains arcs of the form $00a\rightarrow10a\rightarrow11a$ for all choices of $a$. In particular, these are the arcs that were \lq\lq inherited\rq\rq from $\F_3$ and copied at each iterative step. Note that these arcs form chains of length 2.

In each iterative step, we add a bridge arc of type $01\bar0 10\rightarrow01\bar0 11$. In subsequent iterations, this arc will become an arc between vertices $01\bar0 10 b\rightarrow01\bar0 11 b$, where $b$ is any string of the appropriate length. These arcs were added between isolated vertices of $\F_{n-1}$, so they form chains of length 1.
We formalize this in the following lemma.  Let $\bar0^k$ be the all-zero string of length $k$, where $\bar0^0$ is the empty string.

\begin{lem}
\label{arc_size}
    Let $n\geq 3$. All chains in $\F_n$ contain either 0, 1, or 2 arcs. 
    \begin{itemize}
        \item[(i)] For any bit string $a$ of length $n-2$,  $00a\rightarrow10a\rightarrow11a$ forms a chain of length 2, and all chains of length 2 are of this form. 
        \item[(ii)] All chains containing one arc are of the form $01\bar0^k0b\rightarrow 01\bar0^k1b$, where $0\leq k\leq n-3$ and $b$ can be any string of length $n-k-3$. Moreover, for any choice of $k$ and $b$ this arc is in $\F_n$.
        \item[(iii)] The chains containing zero arcs are the isolated vertices, which are $01\bar{0}^{n-3}0$ and $01\bar{0}^{n-3}1$.
    \end{itemize}  
\end{lem}

\begin{proof}
    We prove this by induction.
    
    {\bf Base Case:} Let $n=3$. $\F_3$ is constructed so no chains have 3 or more arcs. The chains containing two arcs are $000\rightarrow100\rightarrow110$ and $001\rightarrow101\rightarrow111$. There are no chains with a single arc and the isolated vertices in $\F_3$ are $010$ and $011$, which fits the formula since $\bar0^{3-3}$ is the empty string. See Figure \ref{Q3zfs}.

    {\bf Induction Step:} Let $n>3$. Assume that $\F_{n-1}$ contains only chains with less than three arcs, all of which have the correct form. 

    Then both copies of $\F_{n-1}$ contribute only chains with 0, 1, or 2 arcs. The bridge arc is constructed between two vertices that were isolated in $\F_{n-1}$. Since neither of these vertices are a part of a chain with other arcs, the bridge arc is therefore a chain with a single arc. None of these chains have 3 or more arcs.

    By induction, the chains that contain two arcs in $\F_{n-1}$ are of the form $00a\rightarrow10a\rightarrow11a$, where $a\in\{0,1\}^{n-3}$ is any bit string of length $n-3$. The chains from the first copy of $\F_{n-1}$ have a 0 appended to the end of each of their vertices' bit strings, and so have the form $00a0\rightarrow10a0\rightarrow11a0$, or $00a'\rightarrow10a'\rightarrow11a'$ where $a'\in\{0,1\}^{n-2}$ is a bit string of length $n-2$ ending in 0. Similarly, the chains from the second copy will have the form $00a''\rightarrow10a''\rightarrow11a''$ where $a''\in\{0,1\}^{n-2}$ is a bit string of length $n-2$ ending in 1. As we do not create any chains with two arcs that cross the matching, these are all of the chains with two arcs. This shows that all chains of length two in $\F_n$ have the required form.

    Now consider chains of length one. By induction, any 
    chain of consisting of one arc has the form $01\bar0^k 0b\rightarrow01\bar0^k1b$ where $0\leq k\leq n-4$ and $b$ is any string of length $n-k-4$. The chain from the first copy of $\F_{n-1}$ has a 0 appended to the end of the bit strings of each of their vertices, and so has the form $01\bar0^k0b'\rightarrow01\bar0^kb'$, where $b'=b0$ has length $n-k-3$. Thus, the arc still has the correct form. Similarly, the chain from the second copy will have the  correct form, if we take  $b'=b1$. 
    The bridge arc, $01\bar{0}^{n-4}10\rightarrow01\bar{0}^{n-4}11$, is a new chain with a single arc since it connects two copies of $01\bar{0}^{n-4}1$, which is isolated in $\F_{n-1}$. Taking $k=n-4$ and $b$ is the empty string, we see that the bridge arc also has the desired form.
    Moreover, the argument shows that for any $0\leq k\leq n-3$ and for any string $b'$ of length $n-k-3$, $01\bar0^k0b'\rightarrow01\bar0^kb'$ is an arc that forms a chain of length 1.

   Finally, we consider the isolated vertices. By induction, the isolated vertices in $\F_{n-1}$ are
   $01\bar{0}^{n-4}0$ and $01\bar{0}^{n-4}1$. The two copies of $01\bar{0}^{n-4}1$ are connected by the bridge arc in $\F_n$, so are no longer isolated. This leaves only $01\bar{0}^{n-4}00$ and $01\bar{0}^{n-4}01$, the two copies of $01\bar{0}^{n-4}0$, as isolated vertices in $\F_n$. These two vertices have the correct form.
   
    So, all chains in $\F_n$ are of the desired form.

\end{proof}

We are able to organize vertices into classes based on the first two digits of their bit string. We will call vertices of the form $00a$ 00-vertices, vertices of the form $10a$ 10-vertices, and so on. Similarly, we will differentiate arcs by the vertex at the tail of the arc. So arcs of the form $00a\rightarrow10a$ will be 00-arcs, $10a\rightarrow11a$ will be 10-arcs, and $01\bar0^k0b\rightarrow01\bar0^k1b
$ will be 01-arcs. Given any vertex bit string, we are able to determine whether there is an arc into or out of that vertex.

There are a few facts about the structure of this twisted hypercube that will prove useful in the following section.

\begin{obs}
\label{twistedge}
    By definition, twisted edges are only ever incident to 10-vertices and 01-vertices.
\end{obs}

\begin{obs}
\label{newtwist}
    The twisted edges in the matching between copies of \newline$(\Q_{n-1}, \F_{n-1})$ are incident to a 10-vertex
    and an isolated vertex in $\F_n$.
\end{obs}

\begin{lem}
\label{num_arcs}
    $\F_n$ contains $2^{n-1}+2^{n-3}-1$ arcs.
\end{lem}

\begin{proof}
    We prove this by induction on $n$. 
    
    {\bf Base Case:} Let $n=3$. Then $\F_n$ contains $4=2^2+2^0-1$ arcs. See Figure \ref{Q3zfs}.

    {\bf Induction Step:} Let $n>3$, and assume, by induction, that $|\F_{n-1}|=2^{n-2}+2^{n-4}-1$. By Lemma \ref{rel_num_arcs}, we have that $|\F_n|=2|F_{n-1}|+1$. Then
    \begin{align*}
        |\F_n|=&2|F_{n-1}|+1\\
            =&2(2^{n-2}+2^{n-4}-1)+1\\
            =&2^{n-1}+2^{n-3}-2+1\\
            =&2^{n-1}+2^{n-3}-1,
    \end{align*}
    as required.
\end{proof}

By Equation \ref{arc_to_zf}, we see that Lemma \ref{num_arcs} is equivalent to saying that $(\Q_n,\F_n)$ corresponds to an initial set of blue vertices of size $2^{n-1}-2^{n-3}+1$ in $\Q_n$.

Note that in our construction, we begin with only 00- and 10-arcs. We then copy these arcs and add a bridge arc, which is a 01-arc. So, at no point do we introduce a 11-arc, which leads to our next observation.

\begin{obs}
\label{no_11_arcs}
    No arc has a 11-vertex as a tail.
\end{obs}

\section{Proof of Main Result}

\label{proof}

Recall from Definition \ref{ctdef} that a chain twist in $(G,\F)$ is a cycle, $C$, in $G$ that contains no consecutive non-arcs in $C$. We can alter this definition to not require the underlying set of vertices to form a cycle, but just a path.

\begin{defn}
\label{ctp}
    Given a graph $G$ and arc set $\F$, a {\it chain twist path} in $\F$ is a path, $v_1v_2\dots v_k$ with the property
    $$(v_i,v_{i+1})\notin\F\Rightarrow (v_{i-1},v_i)\in\F, \mbox{and } (v_{i+1},v_{i+2})\in\F,$$
    for all $1<i<k$.
\end{defn}

So a chain twist path is a path without consecutive non-arcs. It can be thought of as a section of a chain twist. We note an important consequence of this definition.

\begin{prop}
\label{ctp_iso}
    Let $(G,\F)$ be a graph $G$ and arc set $\F$. Let $P=v_1v_2\dots v_k$ be a path in $G$. If $v_i$, $2\leq i\leq k-1$, is an isolated vertex in $\F$, then $P$ is not a chain twist path. 
\end{prop}

\begin{proof}
    Assume that $P$ is a chain twist path, and $v_i$ is isolated in $\F$ for some $2\leq i\leq k-1$. Since $v_i$ is isolated, there are no arcs into $v_i$, and so $v_{i-1}v_i\notin\F$. By Definition \ref{ctp}, this implies that $v_iv_{i+1}\in\F$, but this contradicts $v_i$ being isolated, completing the proof.  
\end{proof}

So we cannot include isolated vertices of an arc set in the interior of any chain twist path, and by extension, any chain twist.

\begin{cor}
    Let $(G,\F)$ be a graph $G$ and arc set $\F$. If $v_1v_2\dots v_kv_1$ is a chain twist, then $v_i$ is not isolated in $\F$, for all $i$. 
\end{cor}

In \cite{chaintwist}, it was shown that an arc set with a closed walk that satisfies the property (\ref{chainprop}) from Definition \ref{ctdef} can be reduced to a chain twist. We will use this fact in the following results.

\begin{thm}
\label{chainwalk}
    \cite{chaintwist} Let $G$ be a graph with arc set $\F$. If $G$ contains a closed walk that satisfies the property (\ref{chainprop}), then $\F$ contains a chain twist.
\end{thm}

It was shown in \cite{minrank} that the zero forcing number of a Cartesian product is bounded above by the zero forcing number and size of the factors. We will offer a new proof of this known result in the context of forcing arc sets as it demonstrates a technique used in further results. 

\begin{prop}
\label{cartprod}
    \cite{minrank} Let $G=(V,E)$ be a graph with a forcing arc set $\F$, and $H=(U,E')$ a graph. Then $G\square H$ has a forcing arc set of size $|\F||H|$.
\end{prop}

\begin{proof}
    Consider $G\square H$, and the arc set $\F^\square$ formed as follows
    $$(v_i,u_r)(v_j,u_s)\in\F^\square\iff u_r=u_s\text{ and }v_iv_j\in\F.$$
    So $\F^\square$ consists of a copy of $\F$ in each copy of $G$ in $G\square H$. Now, for the sake of contradiction, assume that $\F^\square$ has a chain twist $C$. 
    
    If $C$ includes no edge of the form $(v,u_r)(v,u_s)$, then $C$ is contained in a single copy of $G$. Therefore, $\F$ contains a chain twist which is a contradiction. 

    So $C$ must include edges of the form $(v_j,u_r)(v_j,u_s)$. We know, by the construction of $\F^\square$, that none of these edges are arcs. Then Definition \ref{ctdef} tells us that these edges must be preceded and followed by arcs, which exist only in the copies of $G$. Consider the vertices of $C$ and form a new sequence $W=v_1,v_2,\dots,v_k,v_1$ that consists of the first element of each vertex in $C$, where some vertices will be listed consecutively. Note that the same vertex will never be listed three times consecutively, as this equates to traversing two edges from $H$ in a row which contradicts $C$ being a chain twist. By Definition \ref{ctdef} and the construction of $\F^\square$, whenever $\dots v_{i-1},v_i,v_i,v_{i+1}\dots$ appears in $W$, $v_{i-1}v_i$ and $v_iv_{i+1}$ are both arcs in $\F$. This means that removing these repeated vertices will result in a new sequence that satisfies property (\ref{chainprop}). We now form the sequence $W'$, which is $W$ but whenever the same vertex is listed consecutively, we remove the second instance of that vertex. The sequence $W'$ is a closed walk in $G$. Since every edge in $W'$ corresponds to an edge in $C$, $W'$ is a closed chain twist walk in $G$. By Theorem \ref{chainwalk}, $W'$ can be reduced to a chain twist, and so $\F$ must contain a chain twist, contradicting our assumption.

    So $(G\square H,\F^\square)$ cannot contain a chain twist, and $\F^\square$ is a forcing arc set. 
\end{proof}

So if we have a Cartesian product of graphs and a set of arcs repeated in copies of one of the factors, any chain twist can be projected onto the factor. This result can be applied immediately to $(\Q_n,\F_n)$ to tell us information about any chain twists, should they exist.

\begin{prop}
\label{chainarc}
    Let $\Q_n$, $n\geq4$, be a minority cube with arc set $\F_n$. If $\F_n$ has a chain twist, $C$, and $\F_{n-1}$ does not, then $C$ contains the bridge arc.
\end{prop}

\begin{proof}
    Let $C$ be a chain twist in $\F_n$. Assume that the bridge arc is not in $C$. If $C$ contains no matching edges, then $C$ is contained in one copy of $(\Q_{n-1},\F_{n-1})$. But this contradicts $\F_{n-1}$ not having a chain twist. So $C$ must contain matching edges. All available matching edges are not arcs, so $C$ is of the form $v_0,v_1,\cdots, x, y,\cdots,v_{k-1},v_0$, where $x$ and $y$ are twin vertices connected by an edge, and $v_{k-1}\rightarrow v_0$. Assume $x-y$ is a twisted edge. Then one of $x$ or $y$ is an isolated vertex in $\F_n$, and $C$ cannot be a chain twist. Therefore $x-y$ cannot be a twisted edge. Now we consider a new sequence $C'=v_0',v_1',\cdots, x', y',\cdots,v_{k-1}',v_0'$, where $v_i'$ is the vertex $v_i$ with the final bit in the binary expansion removed. This means that $x'=y'$ as $x$ and $y$ are twin vertices and do not share a twisted edge. At the point where $C$ crosses the matching edge, $xy$, $C'$ has a repeated element in its sequence. Define a new sequence $W$ to be $C'$ where we remove the second instance of each repeated element. Then $W$ is a closed walk in $\Q_{n-1}$. As we have only removed edges that are not arcs $\F_n$, $W$ is a closed chain twist walk in $\F_{n-1}$. By Theorem \ref{chainwalk}, this can be reduced to a chain twist. This gives us a chain twist in $\F_{n-1}$ contradicting our assumption.

    Therefore, $C$ must contain the bridge arc, as required.
\end{proof}

Knowing that any chain twist must not only cross over the matching, but must do so at the bridge arc, greatly limits the possible chain twists we must consider in our twisted hypercubes. The following lemma demonstrates the scope of this restriction.

\begin{lem}
\label{ctp_lemma}
    Let $\Q_n$, $n\geq4$, be a minority cube with forcing arc set $\F_n$. Then any chain twist path that begins at the bridge arc, $a_b$, and crosses the matching contains at most one arc after the matching edge.
\end{lem}

\begin{proof}
    Let $P$ be a chain twist path that begins at $a_b$ and crosses the matching at another edge, $e=x-x'$. Note that $e$ cannot be an arc, since, by construction, the bridge arc is the only arc crossing the matching. If $e$ is a twisted edge, then by Observation \ref{newtwist}, $e$ is incident to an isolated vertex in $\F_n$. Since $e$ is not an arc but is part of a chain twist path, $x$ must be the endpoint of an arc. Thus $x'$ must be an isolated vertex in $\F_n$. Therefore, $P$ cannot contain any arc after $e$.
    
    Next, assume that $e$ is not a twisted edge. So $x$ and $x'$ are twins that differ in only the final bit. By Definition \ref{ctp}, since $e$ is not an arc, it must be preceded by an arc. Let $y$ be the tail of the arc into $x$. We will consider all possible arcs before $e$ in $P$ by the classes of arcs separately. 

    Case I: Assume that $y\rightarrow x$ is a 10-arc. By Lemma \ref{arc_size}, $x$ is a 11-vertex. Therefore, $x'$ is also a 11-vertex. By Observation \ref{no_11_arcs}, there are no arcs with $x'$ as a tail, and so there are no arcs after $e$ in $P$.

    Case II: Assume that $y\rightarrow x$ is a 01-arc. 
    Then, by Lemma \ref{arc_size}, this arc is of the form $01\bar0^k0b\rightarrow 01\bar0^k1b$, where $0\leq k\leq n-3$ and $b$ is any string of length $n-k-3$. 
    Note that when $k=n-2$, then $y\rightarrow x$ is the bridge arc $01\bar00\rightarrow 01\bar01$. Then $y=x'=01\bar01$, but if $y\rightarrow x$ is an arc in $\F_n$, then we cannot use $x-y$ as an edge in the chain twist path. 
    
    So we consider the case when $k\leq n-1$. Then  $x=01\bar0^k1b$,
    where $b$ is not empty. 
    Since $x'$ is the twin of $x$,  $x'=01\bar0^k1b' $, where $b$ and $b'$ differ only in the last bit. 
    By Lemma \ref{arc_size}, this implies that $x'$  
    is the head of another 01-arc and has no arc out. So, there are no arcs after $e$ in $P$.

    Case III: Assume that $y\rightarrow x$ is a 00-arc. So for some $c\in\{0,1\}^{n-2}$, $y=00c$ and $x=10c$. 

    The first property to note is that in $P$, $y$ can only be reached from the head of a 01-arc. Namely, let $z$ be the vertex preceding $y$ in $P$. All neighbours of $y$ in $\Q_n$ are either 00-vertices, $10c=x$, and $01c$, by Observation \ref{twistedge}. Since $P$ is a path, $z\not= x$. If $z$ is a 00-vertex, then $z-y$ is an edge. As none of the 00-vertices are the head of an arc, the only possible chain twist path to $y$ is through $z=01c$. Therefore $z-y$ is an edge, so this requires $01c$ to be the head of an arc, by Definition \ref{ctp}, as claimed. 
    So for some $k\leq n-3$ and $b\in \{ 0,1\}^{n-k-3}$, $c=\bar{0}^k0b$. 

    If $k=n-3$, then $y\rightarrow x$ is $00\bar0^k\rightarrow10\bar0^k1$. By Observation \ref{newtwist}, the twin $x'=10\bar{0}^k0$ of $x=10\bar0^k1$ is an isolated vertex in $\F_n$. Therefore, by Proposition \ref{ctp_iso}, there are no further arcs following $e$ in $P$.

    If $k<n-3$, then $y\rightarrow x$  is $00\bar{0}^k0b\rightarrow10\bar{0}^k0b$. By Observation \ref{newtwist}, the twin to $x=10\bar{0}^k0b$ is $10\bar{0}^k0b'$, where $b$ and $b'$ are non-empty bit strings that differ only in the last bit. So $x'=10\bar{0}^k0b'$ is the tail of a 10-arc. The path $P$ can follow this arc to $11\bar{0}^k0b'$. By Observation \ref{twistedge}, all neighbours of $11\bar{0}^k0b'$ are 11-vertices, or the vertex $01\bar{0}^k0b'$. However, 11-vertices do not have out-arcs, and $01\bar{0}^k0b'$ is the head of a 01-arc, and therefore has no arcs out. See Figure \ref{biglemdemo} for a visualization. Therefore, $P$ cannot be extended beyond the 10-arc starting at $x'$.
     
    This shows that if $y\rightarrow x$ a 00-arc, then $P$ has at most one arc after $e$.
    
    These three cases are all possible constructions of a chain twist path of this form in $(\Q_n,\F_n)$. Each of these chain twist paths contain at most one arc following $e$ in $P$, completing the proof.
\end{proof}

\begin{figure}
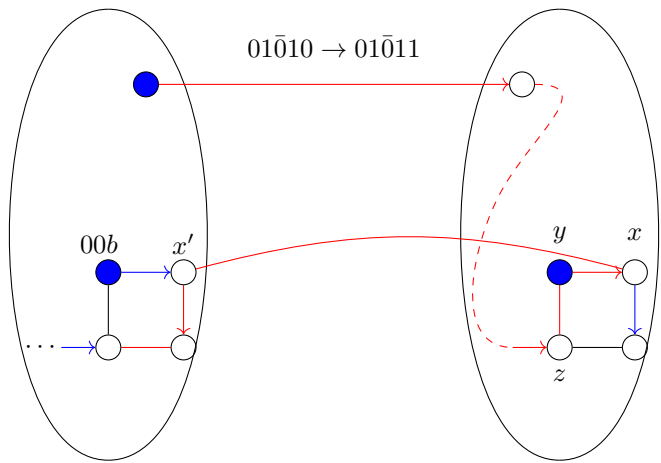

    \centering
    \ctikzfig{thqnoct}
    \caption{If $P$, indicated in red, crosses the matching after a 00-arc, then there is at most one arc in $P$ after this crossing.}
    \label{biglemdemo}
\end{figure}

As any chain twist path in $(\Q_n,\F_n)$ necessarily ends, this implies that there are no possible chain twists that exist with $a_b$.

\begin{cor}
\label{no_bridge}
    If a minority cube contains a chain twist, $C$, then $C$ does not contain the bridge arc.
\end{cor}

\begin{thm}
\label{ct_thm}
    For $n\geq3$, $(\Q_n,\F_n)$ has no chain twist.
\end{thm}
    
\begin{proof}
    We prove this by induction on $n$. 

    {\bf Base Case:} Let $n=3$. By construction, $(\Q_3,\F_3)$ has no chain twist.

    {\bf Induction Step:} Let $n>3$ and assume, by induction, that $(\Q_{n-1},\F_{n-1})$ does not have a chain twist. 
    
    Assume that $(\Q_n,\F_n)$ contains a chain twist. Then, by Lemma \ref{chainarc}, the chain twist must contain the bridge arc, $a_b$. This contradicts Corollary \ref{no_bridge}. So $(\Q_n,\F_n)$ contains no chain twist, completing the proof.
\end{proof}

We are now able to prove Theorem \ref{zfbound}.

{\it Proof of Theorem \ref{zfbound}.} By Theorem \ref{ct_thm}, $\F_n$ corresponds to a successful zero forcing process on $\Q_n$. By Observation \ref{num_arcs}, this zero forcing set has the desired size. \begin{flushright}
    $\square$
\end{flushright}

\section{Future Work}

In this paper, we defined a new family of twisted hypercube called the minority cube and provided an upper bound on the zero forcing number of these graphs. No lower bound is given, but we conjecture that this is the best possible for this family of graphs.

\begin{conj}
    The zero forcing number of the $n$-dimensional minority cube is $2^{n-1}+2^{n-3}-1$.
\end{conj}

We were able to achieve a zero forcing set of this size by performing only one twist at each step of the construction. We conjecture that this is the lowest zero forcing number that can be achieved with a single twist.

\begin{conj}
    Given a family of twisted hypercubes, where at each step in the construction only one twist is performed, the zero forcing number is at least that of the minority cube of the same dimension. 
\end{conj}

We do not know whether performing multiple twists at each step will improve the efficiency of the zero forcing process, so we leave this as an open question.

\begin{ques}
    Is there another family of twisted hypercubes with smaller zero forcing number than the minority cube?
\end{ques}

\bibliographystyle{plain}
\bibliography{refs}{}

\begin{thebibliography}{10}

\bibitem{np-hard}
Ashkan Aazami.
\newblock {\em Hardness results and approximation algorithms for some problems on graphs}.
\newblock PhD thesis, University of Waterloo, 2008.

\bibitem{johngrassham}
Aida Abiad, Robin Simoens, and Sjanne Zeijlemaker.
\newblock On the diameter and zero forcing number of some graph classes in the {J}ohnson, {G}rassmann and {H}amming association scheme.
\newblock {\em Discrete Applied Mathematics}, 348:221--230, 2024.

\bibitem{minrank}
{AIM Minimum Rank – Special Graphs Work Group}.
\newblock Zero forcing sets and the minimum rank of graphs.
\newblock {\em Linear Algebra and its Applications}, 428(7):1628--1648, 2008.

\bibitem{data_qbe}
Redhwan Al-amri, Raja~Kumar Murugesan, Mubarak Almutairi, Kashif Munir, Gamal Alkawsi, and Yahia Baashar.
\newblock A clustering algorithm for evolving data streams using temporal spatial hyper cube.
\newblock {\em Applied Sciences}, 12(13), 2022.

\bibitem{bal2018zeroforcingnumberrandom}
Deepak Bal, Patrick Bennett, Sean English, Calum MacRury, and Paweł Prałat.
\newblock Zero forcing number of random regular graphs, 2018.

\bibitem{zfpar}
Francesco Barioli, Wayne Barrett, Shaun~M. Fallat, H.~Tracy Hall, Leslie Hogben, Bryan Shader, P.~{van den Driessche}, and Hein {van der Holst}.
\newblock Zero forcing parameters and minimum rank problems.
\newblock {\em Linear Algebra and its Applications}, 433(2):401--411, 2010.

\bibitem{zfpmu}
Boris Brimkov, Caleb~C. Fast, and Illya~V. Hicks.
\newblock Computational approaches for zero forcing and related problems.
\newblock {\em European Journal of Operational Research}, 273(3):889--903, 2019.

\bibitem{zfqcp}
Daniel Burgarth, Domenico D'Alessandro, Leslie Hogben, Simone Severini, and Michael Young.
\newblock Zero forcing, linear and quantum controllability for systems evolving on networks.
\newblock {\em IEEE Transactions on Automatic Control}, 58(9):2349--2354, 2013.

\bibitem{chaintwist}
Ben Cameron, Jeannette Janssen, Rogers Matthew, and Zhiyuan Zhang.
\newblock An approximation algorithm for zero forcing.
\newblock {\em arXiv preprint arXiv:2402.08866}, 2024.

\bibitem{mythesis}
Peter Collier.
\newblock Zero-forcing processes on proper interval graphs and twisted hypercubes.
\newblock Master's thesis, Dalhousie University, 2023.

\bibitem{Davila2020}
R.~Davila and M.A. Henning.
\newblock Zero forcing in claw-free cubic graphs.
\newblock {\em Bull. Malays. Math. Sci. Soc.}, 43:673–--688, 2020.

\bibitem{qbe_traffic}
Mario di~Bernardo, Elisa Maini, Antonio Manzalini, and Nicola Mazzocca.
\newblock Traffic dynamics and vulnerability in hypercube communication networks.
\newblock In {\em 2014 IEEE International Symposium on Circuits and Systems (ISCAS)}, pages 2221--2224, 2014.

\bibitem{He2024}
Mengya He, Huixian Li, Ning Song, and Shengjin Ji.
\newblock The zero forcing number of claw-free cubic graphs.
\newblock {\em Discrete Applied Mathematics}, 359:321--330, 2024.

\bibitem{twistcube}
Peter A.~J. Hilbers, Marion R.~J. Koopman, and Jan~L.A. van~de Snepscheut.
\newblock The twisted cube.
\newblock In J.~W. de~Bakker, A.~J. Nijman, and P.~C. Treleaven, editors, {\em PARLE Parallel Architectures and Languages Europe}, pages 152--159, Berlin, Heidelberg, 1987. Springer Berlin Heidelberg.

\bibitem{hogben2022inverse}
L.~Hogben, J.C.H. Lin, and B.L. Shader.
\newblock {\em Inverse Problems and Zero Forcing for Graphs}.
\newblock Mathematical Surveys and Monographs. American Mathematical Society, 2022.

\bibitem{Kalinowski19}
Thomas Kalinowski, Nina Kam\v{c}ev, and Benny Sudakov.
\newblock The zero forcing number of graphs.
\newblock {\em SIAM Journal on Discrete Mathematics}, 33(1):95--115, 2019.

\bibitem{zfiep}
Franklin~H.J. Kenter and Jephian C.-H. Lin.
\newblock A zero forcing technique for bounding sums of eigenvalue multiplicities.
\newblock {\em Linear Algebra and its Applications}, 629:138--167, 2021.

\bibitem{LIANG202381}
Yi-Ping Liang and Shou-Jun Xu.
\newblock On graphs maximizing the zero forcing number.
\newblock {\em Discrete Applied Mathematics}, 334:81--90, 2023.

\bibitem{quant_qbe}
Edric Matwiejew, Jason Pye, and Jingbo~B Wang.
\newblock Quantum optimisation for continuous multivariable functions by a structured search.
\newblock {\em Quantum Science and Technology}, 8(4):045013, 2023.

\bibitem{dir_np-hard}
Maguy Trefois and Jean-Charles Delvenne.
\newblock Zero forcing number, constrained matchings and strong structural controllability.
\newblock {\em Linear Algebra and its Applications}, 484:199--218, 2015.

\bibitem{thqbe}
Xuding Zhu.
\newblock A hypercube variant with small diameter.
\newblock {\em Journal of Graph Theory}, 85(3), 2016.

\end{thebibliography}

\end{document}